\documentclass[11pt]{amsart}

\usepackage{ 
amssymb 
}

\date{\today} 

\newtheorem{theorem}{Theorem}
\newtheorem{definition}{Definition}
\newtheorem{corollary}{Corollary}
 
\theoremstyle{remark} 
 
\DeclareMathOperator{\WF}{WF} 
\DeclareMathOperator{\dist}{dist} 
 
\newcommand{\eps}{\varepsilon}

\newcommand{\R}{\mathbf{R}}
 
\renewcommand{\r}[1]{(\ref{#1})} 
\newcommand{\PDO}{$\Psi$DO} 
\newcommand{\be}[1]{\begin{equation}\label{#1}} 
\newcommand{\ee}{\end{equation}} 

\renewcommand{\i}{\mathrm{i}} 
 
\newcommand{\bo}{\partial M} 

\title[Recovery of a metric from the DN map]
{On the stable recovery of a metric from\\ the hyperbolic DN map with incomplete data} 

\author[P. Stefanov]{Plamen Stefanov}
\address{Department of Mathematics, Purdue University, West Lafayette, IN 47907}
\thanks{First author partly supported by a NSF  Grant DMS-1301646}

\author[G. Uhlmann]{Gunther Uhlmann}
\address{Department of Mathematics, University of Washington, Seattle, WA 98195,  Department of Mathematics University of Helsinki, Finland and Institute for Advanced Study of the Hong Kong University of Science and Technology}
\thanks{Second author partly supported by NSF, a Simons Fellowship and the Academy of Finland}

\author[A. Vasy]{Andras Vasy}
\address{Department of Mathematics, Stanford University, Stanford  CA 94305}
\thanks{Third author partly supported by NSF Grant CMG-1025259 and
 DMS-1361432 }

\begin{document} 

\begin{abstract}
We show that given  two hyperbolic Dirichlet to Neumann maps associated to two Riemannian metrics of a Riemannian manifold with boundary which coincide near the boundary are close then the lens data of the two metrics is the same. As a consequence, we prove uniqueness of recovery a conformal factor (sound speed) locally under some conditions on the latter. 
\end{abstract}

\maketitle 
\section{Introduction} 

Let $(M, g)$ be a Riemannian manifold with smooth boundary. 
Let $\Delta_g$ be the Laplace-Beltrami operator on $M$. In local coordinates, $g$ is represented as a symmetric positive definite matrix $g=(g_{ij})$.
The Laplace-Beltrami operator is given in local coordinates by 
\[
\Delta_g = (\det g)^{-\frac12}  \frac\partial{\partial x_i}
(\det g)^{\frac12}g^{ij}\frac\partial{\partial x_j},
\]
where $(g^{ij}) = (g_{ij})^{-1}$, $\det g = \det(g_{ij})$. Let $A$ be a first order differential operator with smooth coefficients. 
Consider the following initial boundary value problem
\begin{equation}   \label{11}
\left\{
\begin{array}{rcll}
(\partial_t^2 - \Delta_g-A)u &=&0 &  \mbox{in $(0,T)\times M$,}\\
u|_{t=0} = \partial_t u|_{t=0} &=& 0 & \mbox{in $M$,}\\
u|_{(0,T)\times \bo} &=&f, &
\end{array}
\right.               
\end{equation}
where $T>0$ is fixed, $f\in H_{(0)}^1((0,T)\times\bo)$, and the subscript $(0)$ indicates that the functions in that space vanish for $t=0$.
Denote by $\nu=\nu(x)$ the unit outer normal to $\bo$ at $x\in\bo$. Define the hyperbolic Dirichlet-to-Neumann (DN) map $\Lambda_{g}$ by
$$
\Lambda_{g} f := \partial_\nu u \big|_{(0,T)\times\bo},
$$
where $\partial_\nu$ is the exterior unit (in the metric $g$) normal derivative. 
The problem we study is if we can determine  $g$, up to an isometry, in a stable way from the DN map. Uniqueness of such recovery has been established in \cite{BelishevK92}, using the boundary control method developed by \cite{Belishev_87} that relies in the fundamental unique continuation result of Tataru \cite{tataru95}. In this article we do not consider the determination of $A$. See \cite{KathcalovKL-book} and the references there for results in this direction.
H\"older type stability estimates for a generic class of simple metrics were established in \cite{SU-IMRN}. For the case of simple metrics with lower order terms stability estimates were given in \cite{Carlos_12} modulo the appropriate gauge transformations of the lower order terms,  see also \cite{BellassouedDSF}. 
We recall that simplicity of a metric means that the boundary is strictly convex and that the metric has no conjugate points.

The goal of this note is to study stability for
possible non-simple metrics assuming we know $\Lambda$ locally. At present, we restrict our attention to recovery of the principal symbol which dictates the geometry.  We show that in Theorem 3.1 that if the DN maps associated
to to Riemannian metrics are close enough in their natural norms, and $g$ is known in a neighborhood of $\bo$, then its lens relation is uniquely determined (exactly), and the problem is then reduced to uniqueness (not stability) of the lens rigidity one. We also prove a local version of this result. The lens rigidity problem, and the closely related boundary rigidity problem are well studied, see, e.g., \cite{Croke05,Croke_scatteringrigidity,CrokeH02, LassasSU, Mu2,PestovU, SU-JAMS, SU-lens, SUV_localrigidity} and for some classes of metrics,  $g$ can be recovered  up to isometry. The last two works deal with non-simple metrics. 

A result of a similar type for simple metrics was proved recently in \cite{BaoZhang}. In fact, they considered non-simple metrics under a certain  condition for the conjugate points but it is not known if there is a metric for which this condition holds.  In contrast to the earlier works however, it was noticed in \cite{BaoZhang} that closedness of the DN maps leads to exact equality of the metrics (up to an isometry). We explain the reason for this below. We want to point out that there are two factors which make such a result possible: the assumption that we know $g$ near the boundary, and the fact that we work with the standard norm of the DN map. In fact, those results show that this norm may not be natural when studying stability questions. 

We also prove in section 4 new stability for metrics in the same conformal class using Theorem 3.1
and the paper \cite{SUV_localrigidity} of the authors. More precisely, if $g_0$ is given,  and $g=c^{-2}g_0$, is in the same conformal class of $g_0$ we want to recover $c$ in a stable way. In this case there is no isometry involved. One of the reasons for considering the operator $A$ is to include in the results the acoustic equation $(\partial_t^2-c^2\Delta_{g_0})$ as well; then  we can take $g=c^{-2}g_0$ but $A$ is non-trivial. In this case $c$ represents physically  the sound speed of the medium. For simple metrics in the same conformal class H\"older type stability estimates were proven in \cite{SU-IMRN}. We remark is that in this
case the result is for all the sound speeds not just a generic set. For other type of stability estimates for the sound speed see \cite{ShitaoO}.

\section{Preliminaries} 
By \cite{KathcalovKL-book,LasieckaLT}, for any $f\in H_{(0)}^1((0,T)\times\bo)$,   the problem \r{11} has unique solution $u\in C([0,T]\times H^1(M))\cap C(  [0,T]\times L^2(M) ) $ continuously depending on $f$; and moreover, 
\be{2.1}
\Lambda :  H_{(0)}^1((0,T)\times\bo)\to  L^2((0,T)\times\bo)
\ee
 is continuous. 
Here, $H_{(0)}^1((0,T)\times\bo)$ is the closed subspace of $H^1((0,T)\times\bo)$ consisting of functions vanishing at $t=0$. 
 One cannot expect that it would  depend continuously on $g$ and in particular on $c$ when $g=c^{-2}g_0$ with $g_0$ fixed. To illustrate this, consider the following well known simple example first: the shift operator $U_cf(x) = f(x-c)$ with $c$ a constant. This can be considered as the dynamics at time $t=1$ associated to the transport operator $\partial_t+ c\partial_x$. Then $U_c$ is a unitary operator on $L^2(\R)$ for any $c$ but $U_c$ is not a continuous family of operators in the operator norm. To see this, one constructs functions $f_n$ with norm one with support shrinking to $0$. Then   actually, $\|U_{c_1}-U_{c_2}\|=2$ for $c_1\not=c_2$. In particular, $\|U_{c_1}-U_{c_2}\|<2$ implies $c_1=c_2$.
This can be extended to elliptic FIOs of order $1$, for example, with a canonical relation a graph of a diffeomorphism. To prove a similar result, one constructs a sequence of distributions asymptotically  concentrated at a single point in the phase space. In short, this is the idea of the proof of Theorem~\ref{thm_easy2} below. 

Going back to $\Lambda$, it is an elliptic FIO of order one microlocally for covectors which are projections of characteristic  covectors pointing strictly into $\R\times M$, and away from any neighborhood of $t=0$, so we are in a similar situation. On the other hand, near $t=0$ and $x=y\in\bo$, and directions pont strictly into $M$, $\Lambda$ is a \PDO\ of order $1$ depending continuously on $c$. 
Near tangent directions however,  the two Lagrangians intersect. 
The assumption that $g$ is known near $\bo$ eliminates the need to study that intersection, where most of the difficulties lie. It also eliminates the need to study the diagonal (i.e., the transversal geodesics, at $t=0$). The latter problem is well understood in terms of uniqueness and stability, even for non-convex boundaries, see  \cite{SU-lens, SU-JAMS} and the references there.

We define the scattering relation next. Assume for convenience that $\bo$ is strictly convex, see also \cite{SU-lens} for the more general case. 
Let $\partial_\pm SM$ be the open sets of all vectors $(x,v)$ with $x\in\bo$, $v$ unit in  the metric $g$,    pointing outside/inside $M$. 
We define the lens relation
\[
L: \partial_- SM \longrightarrow \partial_+SM 
\]
in the following way: for each $(x,v)\in \partial_- SM$, $L(x,v)=(y,w)$, where $y$ is the exit point, and $w$ the exit direction, if exist, of the maximal unit speed geodesic $\gamma_{x,v}$ in the metric $g$, issued from $(x,v)$. Let 
\[
\ell: \partial_-SM \longrightarrow \R\cup\infty
\]
be its length, possibly infinite.  
The metric $g$ is called non-trapping, if $\ell$ has a finite upper bound. In the latter case, we denote by $T_0(M)$ the least one. 

It is convenient to parameterize $\partial_\pm SM$ by its projection on the unit ball bundle $B\bo$ of the boundary. In either case, the projection uniquely identifies one of the two unit vectors with the given projection by the choice of the sign $\pm$. We will keep the notation $L$ and $\ell$ for the so parameterized quantities, as well. Given $\rho'\in B\bo$, the geodesic issued from $\rho\in \partial_-SM$ having a projection $\rho'$ will be also called the geodesic issued from $\rho'$ and sometimes denoted by $\gamma_{\rho'}$. 

We will identify below vectors and covectors by the metric.

\section{Main results}
\begin{theorem}\label{thm_easy2}
Let $\mathcal{K}\Subset\partial_-SM$ be compact with $\max_{\mathcal{K}}\ell<\infty$. Let $\Gamma_\pm$ be two open subsets of $\bo$ so that 
$\pi(K)\subset \Gamma_-$ and $\pi\circ L(K)\subset \Gamma_+$. Let $\eps>0$ and $T-\eps> 
\max_{\mathcal{K}}\ell$. 
Then  there exists $\delta=\delta(\mathcal{K})>0$ so that if
\be{eq_easy}
\|\Lambda_{g}-  \Lambda_{\tilde g}\|_{H_{0}^1((0,\eps)\times\Gamma_1)\to  L^2((0,T)\times\Gamma_2)}\le \delta_,
\ee
for some other metric $\tilde g$, 
then $L_{g}=L_{\tilde g}$ and $\ell_{ g}= \ell_{\tilde g}$ on $\mathcal{K}$.  
\end{theorem}

\begin{proof}
We will use solutions essentially supported near a single geodesic which may have conjugate points.  In \cite{BaoZhang}, this is done using Gaussian beams. 
We will use the semiclassical calculus  with $h>0$ a small parameter, see e.g., \cite{Zworski_book}. The properties of the solutions then follow from the calculus, and there is no need to construct them explicitly even though we can: on each short enough interval for $t$, those solutions are given by the geometric optics construction with a suitable phase solving the eikonal equation, see e.g., \cite{SU-thermo}. 
 We still need to construct a reflected solution $u_1$ below but the latter needs to be constructed for a small time interval, and the same standard geometric optics construction then works well, see also \cite{SU-thermo_brain}. 

Recall that the following functions
\[
F(z;z_0,\zeta^0) = h(\pi h)^{-n/4} e^{\frac\i{h} ( (z-z_0)\cdot\zeta^0- |z-z_0|^2/2)}.
\]
are called coherent states, for any fixed $(z_0,\zeta^0)\in \R^{2n}$, $\zeta^0\not=0$, see \cite{Zworski_book}. The semiclassical wave front set $\WF_h(F)$ consists of one point only, $(z_0,\zeta^0)$. The extra factor $h$ is chosen so that $1/C<\|F\|_{H^1}<C$. Without that factor, $F$ is unit in $L^2$. 

Fix $\rho_0:= (x_0,\xi^0)\in B^*\bo$, and work in fixed local coordinates in some neighborhood $U_0$ of $x_0$. Set 
\be{f}
f_{\rho_0} = F(t,x; \eps/2,x_0,-1,\xi^0).
\ee
Here, $z=(t,x)$. We can always assume that $f_{\rho_0}$ is localized near $(\eps,x_0)$ with a cut-off function.
This is the coherent state on $\bo$ with $\WF_h(f_\rho)$ at $(t_0,x_0,\tau^0,{\xi^0})= (\eps_0/2,x_0,-1,\xi^0)$. The standard geometric optics construction, see, e.g., \cite{SU-IMRN, SU-thermo_brain}, implies that the solution $u$ of \r{11} with $f=f_{\rho_0}$ has a semi-classical wave front on the zero bicharacteristic issued from $(y_0, \eta^0)$, where $|\eta^0|=1$ (the norm is in the metric $g$), $\eta^0= (\xi^{0},\eta_n^0)$ in boundary normal coordinates, and $\eta_n^0$ is the positive root of $(\eta_n^0)^2 +|\xi^0|^2=1$.  In other words, $\eta^0$ is the outgoing (future pointing) unit covector with projection $\xi^0$. There are no singularities propagating back to $t=0$ because of the zero Cauchy data there. If we replace $\tau_0=-1$ in \r{f} by $\tau_+=+1$, we get a similar singularity with $\eta_n^0<0$ propagating along the  bicharacteristic determined by $-(\xi^0,\eta_n^0)$ (in other words, by $(\xi^0,\eta_n^0)$ but for negative times). We are not going to use the second kind.  

 Assume (as in the theorem) that the geodesic issued from $\rho_0$ at $t=\eps/2$ is non-trapping, i.e., hits $\bo$ again in $\Gamma_2$, transversely, for a finite time $t_1=\ell(\rho_0)$, at $(t_1,x_1)$.  Then the wave front set of $u$ is on the bicharacteristic over that geodesic (see, e.g., \cite{Zworski_book}), until it hits $\bo$, again, then it reflects (see \cite{SU-IMRN}, etc. Let for a moment $u_0$ be as $u$ but propagating outside of $M$ near $(t_1,x_1)$; with some smooth extension of $g$ there. The map $f\mapsto \partial_\nu u_0|_{\R\times\bo}$, for any $f$ with wave front set near $\rho_0$ (notice that tangential directions are avoided) is an $h$-FIO of order $1$ with a canonical relation a diffeomorphism. Let $u_1$ be the reflected solution, i.e., $u_0+u_1=0$ on $\bo$ near $(t_0,x_0)$, and $u_1$ has wave front set obtained by that of $u_0$ by reflection, see  \cite{SU-IMRN,SU-thermo}. Then in $M$, in some neighborhood $U_1$ of  $(t_1,x_1)$ in $\R\times\bo$, $u$ is given by $u=u_0+u_1$ and  the map $f\mapsto \partial_\nu u|_{U_2}$ is an $h$-FIO of order $1$  again, with principal symbol equal to twice that of the FIO above. Therefore, if we choose  $T_1>t_1$ so that it is less than the time $t_2$ for which the second reflection takes place,  $\WF_h\left(u_{[0,T_1]\times\bo}\right)$ lies  over $U_0\cup U_1$. 

So far $\rho_0 =(x_0,\eta^0)$ was fixed and the construction is not invariant under change of variables. Let us now vary $\rho$ over the whole $\mathcal{K}$. For $\rho$ near each such $\rho_0$, we have an uniform estimate from below for $\|\Lambda f_\rho\|_{L^2(U_1)}$ for $0<h< h_0$ with some $h_0>0$ with $U_1 =U_1(x_0,\eta^0)$ as above. Then we can chose a finite sub-cover with the same property and define $f_\rho$ globally by a partition of unity. 
Therefore, there exists $\delta_0>0$ so that
\be{3.1}
2\delta_0 \le \|\Lambda f_\rho\|_{L^2((\eps,T)\times \bo)},\quad \forall \rho\in \mathcal{K}, \quad 0<h< h_0.
\ee
 Notice that above, we excluded the  interval $(0,\eps)$.

Let $\tilde u$ be constructed as above, but related to $\tilde g$. Under the assumptions of the theorem, using finite speed of propagation, $u(t,x)=\tilde u(t,x)$ for $t$ near $\eps/2$ and all $x$. Then the normal derivatives near $(\eps/2,x_0)$ equal as well, i.e., $\Lambda_g f_\rho= \Lambda_{\tilde g} f_\rho$ there. Assume that $(L_g,\ell_g)$ and $(L_{\tilde g},\ell_{\tilde g})$ do not coincide in $\mathcal{K}$, which includes the possibility that for some $\rho\in\mathcal{K}$, $\tilde \gamma_\rho$, related to $\tilde g$, is trapping, i.e., $\ell_{\tilde g}(\rho)=\infty$. Then there exists $\rho\in\mathcal U$ so that $(L_g(\rho),\ell_g(\rho)) \not = (L_{\tilde g}(\rho), \ell_{\tilde\rho}(\rho))$, or the latter is undefined. 
Then near $\WF_h(\Lambda_g f_\rho)$, we have $\Lambda_{\tilde g} f_\rho=O(h^\infty)$ and over $(\eps,T)\times \bo$, $\Lambda_g f_\rho$ and $\Lambda_{\tilde g} f_\rho$ have non-intersecting wave front sets if $T$ is chosen larger than $\eps/2+\ell(\rho)$ but smaller than the second time the reflected ray hits $\bo$ again. Then they are orthogonal modulo $O(h^\infty)$. 
Then
\be{3.2}
\begin{split}
\| (\Lambda_g - \Lambda_{\tilde g} )f_\rho\|^2_{L^2((T_0,T)\times \bo)}  &= \| \Lambda_g f_\rho\|^2_{L^2((\eps,T)\times \bo)}+  \|  \Lambda_{\tilde g} f_\rho\|^2_{L^2((\eps,T)\times \bo)} + O(h^\infty)\\
&\ge \| \Lambda_g f_\rho\|^2_{L^2((\eps,T)\times \bo)}  - O(h^\infty)
\end{split}
\ee
This however contradicts \r{3.1} and \r{eq_easy} for $0<h\le h_0(\rho,\tilde g)\ll1$. 
\end{proof}

In particular, with $\Gamma_1=\Gamma_2= \bo$, and 
if $g-\tilde g$ is a priori supported away from $\bo$, then $\delta\ll1$ in \r{eq_easy} implies $L_g=L_{\tilde g}$ and $\ell_g = \ell_{\tilde g}$ globally, which is one of the results in \cite{BaoZhang}.  

Note that the l.h.s.\ of \r{eq_easy} in Theorem~\ref{thm_easy2} can be microlocalized by replacing $\Gamma_\pm$ with open sets in the phase space $B\bo$; then we remove the projection $\pi$ in the formulation of Theorem~\ref{thm_easy2}. To localize there, we can choose suitable zero order \PDO s with principal symbols supported in those sets, and use them as cutoffs. 

\section{Recovery of a sound speed} 

To get  results for  non-simple metrics and partial data, we apply a recent result obtained by the authors in  \cite{SUV_localrigidity}. Assume in this section, that $g=c^{-2}g_0$ with $g_0$ fixed. We showed in \cite{SUV_localrigidity}, that one can recover $c$ near some $p\in\bo$ in $N$, if $g$ is given, by the lens relation known locally near $S_p\bo$. Then we extended this local result globally, assuming the foliation condition below. 
Note that we can have conjugate points under the foliation condition, without a restriction of their type. Also, the global arguments in  \cite{SUV_localrigidity} are based on layer stripping arguments, and work even if recovery of the wave front set from the X-ray transform may not be possible locally. 
  Below, $\tilde M$ is an open extension of $M$.

\begin{definition}\label{def_5.1}
Let $M_0\subset M$ be compact. We say that $M_0$ can be foliated by strictly convex hypersurfaces if there exists 
 a smooth function $\rho: \tilde M\to[0,\infty)$ which level sets $\Sigma_t=\rho^{-1}(s)$, $0\le s\le S$
 with some $S>0$, restricted to $M$, are strictly convex  viewed from $\rho^{-1}(0,s)$ for $g$;   $d\rho\not=0$  on these level sets,   $\Sigma_0\cap M=\emptyset$, and $M_0\subset \rho^{-1}(0,S]$. 
\end{definition}

Examples of foliations are the geodesic balls for metrics of negative curvature, or the surfaces $\{x\in M|\; \dist(x,\bo)=\eps\}$, for $\eps\ll1$ if $\bo$ is strictly convex. Any $M_0$ as above is non-trapping in the sense that any maximal geodesic in $M_0$ is of finite length. We denote the least upper bound of those geodesics by $T_0(M_0)$. 

Then we get the following local result, as consequence of Theorem~{5.2} in \cite{SUV_localrigidity}. Note that we use the uniqueness part of the latter only. 

\begin{corollary}\label{col3}
Let $n=\dim M\ge3$. 
Assume that $M_0\subset M$    can be foliated by strictly convex  hypersurfaces $\rho^{-1}(s)$, $0\le s\le S$. 
 Let $\Gamma\subset \bo$ be  a neighborhood  
 of $\Gamma_0:=\rho^{-1}(0,S)\cap \bo$. 
 Let $c=\tilde c$ in some neighborhood of $\Gamma_0$ in $M$. 
Then  if \r{eq_easy} holds with $g=c^{-2}g_0$, $\tilde g = \tilde c^{-2}g_0$, and 
with $\Gamma_1=\Gamma_2=\Gamma$, and $T_0>T_0(M)$, we have 
\be{stab_global}
c=\tilde c\quad \text{on $M_0$}.
\ee
\end{corollary}

Clearly, we can require above only the geodesics hitting $\Gamma_0$ to be non-trapped, with $T_0$ being an upper limit of their lengths. 

Combining Theorem~\ref{thm_easy2} with \cite{SUV_localrigidity}, we immediately get the following. 
\begin{corollary} \label{cor_easy}
Let $n=\dim M\ge 3$. Assume that $M=M_0\cup M_1$ is relatively open and  can be foliated with strictly convex surfaces with respect to the metric $g$, and $M_1$ is simple. Assume, as above, that $c=\tilde c$ near $\bo$.  If \r{eq_easy} holds  with $\Gamma_1=\Gamma_2=\bo$, $T>T_0(M)$, then $c=\tilde c$. 
\end{corollary}



\begin{thebibliography}{10}

\bibitem{BaoZhang}
G.~Bao and H.~Zhang.
\newblock Sensitivity analysis of an inverse problem for the wave equation with
  caustics.
\newblock {\em arXiv:1211.6220}, 2012.

\bibitem{Belishev_87}
M.~I. Belishev.
\newblock An approach to multidimensional inverse problems for the wave
  equation.
\newblock {\em Dokl. Akad. Nauk SSSR}, 297(3):524--527, 1987.

\bibitem{BelishevK92}
M.~I. Belishev and Y.~V. Kurylev.
\newblock To the reconstruction of a {R}iemannian manifold via its spectral
  data ({BC}-method).
\newblock {\em Comm. Partial Differential Equations}, 17(5-6):767--804, 1992.
\newblock MR1177292.

\bibitem{BellassouedDSF}
M.~Bellassoued and D.~Dos Santos~Ferreira.
\newblock Stability estimates for the anisotropic wave equation from the
  {D}irichlet-to-{N}eumann map.
\newblock {\em Inverse Probl. Imaging}, 5(4):745--773, 2011.

\bibitem{Croke05}
C.~Croke.
\newblock Boundary and lens rigidity of finite quotients.
\newblock {\em Proc. Amer. Math. Soc.}, 133(12):3663--3668 (electronic), 2005.

\bibitem{Croke_scatteringrigidity}
C.~Croke.
\newblock {S}cattering rigidity with trapped geodesics.
\newblock peprint, 2012.

\bibitem{CrokeH02}
C.~Croke and P.~Herreros.
\newblock Lens rigidity with trapped geodesics in two dimensions.
\newblock {\em preprint}, 2012.
\newblock
  http://www.math.upenn.edu/$\sim$ccroke/dvi-papers/SurfaceLensRigidity.pdf.

\bibitem{KathcalovKL-book}
A.~Katchalov, Y.~Kurylev, and M.~Lassas.
\newblock {\em Inverse boundary spectral problems}, volume 123 of {\em Chapman
  \& Hall/CRC Monographs and Surveys in Pure and Applied Mathematics}.
\newblock Chapman \& Hall/CRC, Boca Raton, FL, 2001.

\bibitem{LasieckaLT}
I.~Lasiecka, J.-L. Lions, and R.~Triggiani.
\newblock Nonhomogeneous boundary value problems for second order hyperbolic
  operators.
\newblock {\em J. Math. Pures Appl.}, 65(2):149--192, 1986.

\bibitem{LassasSU}
M.~Lassas, V.~Sharafutdinov, and G.~Uhlmann.
\newblock Semiglobal boundary rigidity for {R}iemannian metrics.
\newblock {\em Math. Ann.}, 325(4):767--793, 2003.

\bibitem{ShitaoO}
S.~Liu and L.~Oksanen.
\newblock A lipschitz stable reconstruction formula for the inverse problem for
  the wave equation.
\newblock {\em Trans. Amer. Math. Soc.}, 2015.

\bibitem{Carlos_12}
C.~Montalto.
\newblock Stable determination of a simple metric, a covector field and a
  potential from the hyperbolic {D}irichlet-to-{N}eumann map.
\newblock {\em Comm. Partial Differential Equations}, 39(1):120--145, 2014.

\bibitem{Mu2}
R.~G. Muhometov.
\newblock On a problem of reconstructing {R}iemannian metrics.
\newblock {\em Sibirsk. Mat. Zh.}, 22(3):119--135, 237, 1981.

\bibitem{PestovU}
L.~Pestov and G.~Uhlmann.
\newblock Two dimensional compact simple {R}iemannian manifolds are boundary
  distance rigid.
\newblock {\em Ann. of Math. (2)}, 161(2):1093--1110, 2005.

\bibitem{SU-JAMS}
P.~Stefanov and G.~Uhlmann.
\newblock Boundary rigidity and stability for generic simple metrics.
\newblock {\em J. Amer. Math. Soc.}, 18(4):975--1003, 2005.

\bibitem{SU-IMRN}
P.~Stefanov and G.~Uhlmann.
\newblock Stable determination of generic simple metrics from the hyperbolic
  {D}irichlet-to-{N}eumann map.
\newblock {\em Int. Math. Res. Not.}, 17(17):1047--1061, 2005.

\bibitem{SU-lens}
P.~Stefanov and G.~Uhlmann.
\newblock Local lens rigidity with incomplete data for a class of non-simple
  {R}iemannian manifolds.
\newblock {\em J. Differential Geom.}, 82(2):383--409, 2009.

\bibitem{SU-thermo}
P.~Stefanov and G.~Uhlmann.
\newblock Thermoacoustic tomography with variable sound speed.
\newblock {\em Inverse Problems}, 25(7):075011, 16, 2009.

\bibitem{SU-thermo_brain}
P.~Stefanov and G.~Uhlmann.
\newblock Thermoacoustic tomography arising in brain imaging.
\newblock {\em Inverse Problems}, 27(4):045004, 26, 2011.

\bibitem{SUV_localrigidity}
P.~Stefanov, G.~Uhlmann, and A.~Vasy.
\newblock Boundary rigidity with partial data.
\newblock {\em  arXiv:1306.2995}, 2013.

\bibitem{tataru95}
D.~Tataru.
\newblock Unique continuation for solutions to {PDE}'s; between {H}\"ormander's
  theorem and {H}olmgren's theorem.
\newblock {\em Comm. Partial Differential Equations}, 20(5-6):855--884, 1995.

\bibitem{Zworski_book}
M.~Zworski.
\newblock {\em Semiclassical analysis}, volume 138 of {\em Graduate Studies in
  Mathematics}.
\newblock American Mathematical Society, Providence, RI, 2012.

\end{thebibliography}

\end{document}